\documentclass{amsart}
\usepackage{amsfonts}
\usepackage{amsmath,amssymb}
\usepackage{amsthm}
\usepackage{amscd}
\usepackage{graphics}
\usepackage{graphicx}

\theoremstyle{remark}{
\newtheorem{Def}{{\rm Definition}}
\newtheorem{Ex}{{\rm Example}}
\newtheorem{Rem}{{\rm Remark}}

}
\theoremstyle{plain}
{

\newtheorem{Thm}{Theorem}

}

\begin{document}
\title[On certain Reeb spaces of 1st derivatives of proper submersions]{Reeb spaces of 1st derivatives of proper submersions of certain classes}
\author{Naoki kitazawa}
\address{Osaka Central Advanced Mathematical Institute (OCAMI) \\
3-3-138 Sugimoto, Sumiyoshi-ku Osaka 558-8585
TEL: +81-6-6605-3103
}
\email{naokikitazawa.formath@gmail.com}
\urladdr{https://naokikitazawa.github.io/NaokiKitazawa.html}
\maketitle
\begin{abstract}
We study the (canonical) 1st derivatives of {\it proper} submersions represented as height functions and belonging to a certain class: a {\it proper} map means a map the preimage of a compact set by which is always compact. We investigate their {\it Reeb spaces}. 
They are the quotient spaces defined by the equivalence relations on the manifolds of the domains where we identify two points in a same connected component of a same level set of them. They have been important since the establishment of theory of Morse functions, in the 20th century, They are in certain tame situations $0$- or $1$-dimensional and graphs naturally. Related facts have been shown by Gelbukh and Saeki in the 2020s for certain proper smooth real-valued functions. 

In non-proper cases, even related explicit theory has been difficult, except some previously given case of the author. Our study is on a new related case. 
\end{abstract}
%【REVISE】 combinatoric ～ is → combinatorial object. It is .
%【REVISE】  such that a point is a vertex if and only if the corresponding connected component of the level set contains some singular points → whose vertex set is the set of all points containing some singular points in the corresponding connected component of the level set .
%【REVISE】 We delete "extending the result before".
\section{Introduction.}
\label{sec:1}
\subsection{What is the Reeb space of a (continuous) real-valued function $c:X \rightarrow \mathbb{R}$.}
\label{subsec:1.1}

Theory of {\it Morse} functions is a classical theory in singularity theory of differentiable maps and related differential topology, and fundamental and important in various geometry. The {\it Reeb space} $R_c$ of a smooth real-valued function $c:X \rightarrow \mathbb{R}$ or more generally, that of a continuous real-valued function $c:X \rightarrow \mathbb{R}$ has been important since the establishment of the theory, in the 20th century. This is the quotient space $X/{\sim c}$ defined by the equivalence relation $\sim c$ on $X$ with the following rule: $x_1 {\sim}_c x_2$ if and only if $x_1$ and $x_2$ are in a same connected component of some preimage $c^{-1}(t)$ ($t \in \mathbb{R}$). See also \cite{reeb}. For {\it Morse} functions, see \cite{milnor1, milnor2} and more generally for {\it Morse-Bott} functions, see \cite{banyagahurtubise}. Related to this, singularity theory of differentiable maps, see \cite{golubitskyguillemin} for example.

Surprisingly, fundamental theory of Morse functions on compact manifolds are still developing. Moreover, theory of Reeb spaces of smooth functions on compact manifolds are still developing actively, in the 2020s. We focus on Reeb spaces rather than Morse functions or nice smooth functions yielding their nice Reeb spaces.

First, for Morse(-Bott) functions on compact manifolds, their Reeb spaces are naturally graphs. Recently, for more general smooth functions on compact manifolds and continuous functions on compact topological spaces, their Reeb spaces are shown to have natural structures of graphs or shown to be more general Hausdorff spaces of dimension at most $1$, in certain nice situations. See \cite{gelbukh1, gelbukh2, gelbukh3, gelbukh4, saeki1, saeki2}. 

In cases their Reeb spaces are homeomorphic to graphs, we can have their {\it Reeb graphs} as specific graphs in certain nice situations (Theorem \ref{thm:1}). For Morse-Bott functions on compact manifolds, we have them and see \cite{izar, martinezalfaromezasarmientooliveira} in addition to \cite{reeb}.

For cases of non-compact spaces ({\it non-proper} maps), related theory is still immature. Some are explicitly discussed. For example, some conditions for their Reeb spaces to be Hausdorff are found for specific cases in \cite{gelbukh3, gelbukh4}. One of explicit cases is discussed in \cite{kitazawa1} by the author and he has recently discussed explicit non-proper cases additionally: regions in the plane ${\mathbb{R}}^2$ surrounded by the graphs of two smooth real-valued functions and natural smooth maps onto them and the compositions with the canonical projection ${\pi}_{2,1}$ are considered. Note that ${\mathbb{R}}^k$ ($\mathbb{R}:={\mathbb{R}}^1$) is the $k$-dimensional Euclidean space and that ${\pi}_{k,k_1}:{\mathbb{R}}^k \rightarrow {\mathbb{R}}^{k_1}$ denotes the canonical projection ${\pi}_{k,k_1}(x)=x_1$ ($x:=(x_1,x_2)$ with $1 \leq k_1<k$). For this, see preprints \cite{kitazawa4, kitazawa5, kitazawa6, kitazawa7} of the author and here, this is first presented in the subsection \ref{subsec:1.3}.

Note that we do not assume non-trivial arguments in unpublished preprints in the present paper. We review some fundamental notions and notation shortly.
\subsection{Fundamental notions and notation.}
\label{subsec:1.2}
\subsubsection{On maps and topological spaces.}
For a map $c:X \rightarrow Y$ and a subset $Z \subset X$, we use $c {\mid}_Z$ for the restriction of $c$ to $Z$. A function $c:X \rightarrow Y$ between topological spaces is {\it proper} if the preimage of $c^{-1}(K)$ of any compact subset $K \subset Y$ is compact. A map between topological spaces is {\it non-proper} if it is not proper. We use ${\overline{Z}}^X$ for the closure of $Z \subset X$ in a topological space $X$.

For a topological space $X$ with topological dimension being defined, $\dim X$ denotes this. A CW or cell complex is of such spaces and smooth or PL manifolds and polyhedra are CW complexes. A {\it graph} (an {\it E-graph}) is a at most $1$-dimensional CW complex which is locally finite and the closure of each $1$-cell in which is homeomorphic to $D^1$ or $S^1$ (resp. $D^1, S^1, \{t \mid \geq 0\}, \mathbb{R} \subset \mathbb{R}$). For  a at most $1$-dimensional CW complex (or more general cell complex), we call a $0$-cell a {\it vertex} and a 1-cell an {\it edge}. The {\it vertex set} of this is the set of all vertices of it.

Last, let $S^k \text{(}D^{k+1}\text{)}:=\{x \in {\mathbb{R}}^{k+1} \mid {\Sigma}_{j=1}^{k+1} {x_j}^2= \text{(resp.} \leq \text{)}1\}$ denote the $k$-dimensional (unit) sphere (resp. disk).
\subsubsection{Singularity of differentiable (smooth) maps and Reeb graphs of smooth functions.} For a differentiable manifold $X$, let $T_p X$ denote the tangent bundle. This is a real vector space of dimension $\dim X$. We can define the tangent bundle $TX$ over $X$. This is a real vector bundle of dimension $\dim X$ whose fiber is isomorphic to each $T_p X$ and which is a differentiable manifold of dimension $2\dim X$. 
For a differentiable map $c:X \rightarrow Y$, a {\it singular} point $p \in X$ of it means a point where the rank of the differential $dc:TX \rightarrow TY$ drops: $dc$ is a bundle homomorphism between the bundles.
Let the set of all singular points of $c$ be denoted by $S(c)$ and we call it the {\it singular set} of $c$. We use "{\it critical}" instead of "{\it singular}" in the case $Y$ is at most $1$-dimensional. A {\it  submersion} $c$ means a differentiable map $c$ with $S(c)$ being empty. A {\it diffeomorphism} means a submersion which is also a homeomorphism and smooth and we can define the equivalence relation on the family of all smooth manifolds and the notion that two smooth manifolds are {\it diffeomorphic}. 

For a continuous function $c:X \rightarrow \mathbb{R}$ and its Reeb space $R_c$, $q_c:X \rightarrow R_c$ is defined and the unique continuous function $\bar{c}:R_c \rightarrow \mathbb{R}$ with $c=\bar{c} \circ q_c$ is defined.
\begin{Thm}[\cite{saeki1}]
\label{thm:1}
For a smooth function $c:X \rightarrow \mathbb{R}$  on a closed manifold with $c(S(c))$ being finite, the $R_c$ is a graph whose vertex set is $q_c(S(c))$.
\end{Thm}
\begin{Def}
The graph $R_c$ in Theorem \ref{thm:1} is the {\it Reeb graph} of $c$. Hereafter, if we give the structure of a complex on $R_c$ by this rule canonically, let us use ${\mathcal{R}}_c$ for $R_c$. 
\end{Def}
For a Riemannian manifold $X$ with $\dim X \geq 1$, we have $UTX$ as a subbundle of $TX$ consisting of all elements in $TX$ of length $1$ with its fiber being $S^{\dim X-1} \subset {\mathbb{R}}^{\dim X}$.
${\mathbb{R}}^m$ is of simplest Riemannian manifolds and on each submanifold $X \subset {\mathbb{R}}^{m}$, its canonical Riemannian metric is induced and in the case $\dim X>0$ we have $UTX$. For a smooth function $c:={\pi}_{m,1} {\mid}_{X}:X \rightarrow \mathbb{R}$ on a smooth submanifold $X \subset {\mathbb{R}}^{m}$, we have a continuous function $c^{\prime,{\pi}_{m,1}}$ on $X$ which is smooth on $X-S(c)$ as follows. We have the unique element $v_{p,c,+} \in T_p X$ for $p \in X-S(c)$ satisfying the following. $v_{p,c,+}$ is orthogonal to all elements of $T_p X$ mapped to the zero vector by $dc$ and by $dc$, $v_{p,c,+}$ is mapped to a non-zero vector of $T_{c(p)} \mathbb{R}$ positively oriented and represented by a positive number $r_{p,c,+}$ canonically. For $p \in S(c)$, let $r_{p,c,+}:=0$. We can define a smooth function $0 \leq {c^{\prime,{\pi}_{m,1}}}(p):=r_{p,c,+} \leq 1$.

\subsection{Our regions in ${\mathbb{R}}^2$ and natural smooth maps onto them.}
\label{subsec:1.3}
Our region is a non-empty, open, and connected set $D_{\{S_j\}_{j=1}^l}$ in ${\mathbb{R}}^2$ and with an $m$-dimensional submanifold $X_{D,{\{S_j\}_{j=1}^l},m} \subset {\mathbb{R}}^{m_0}$ ($m_0 \geq m \geq 2$) having no boundary, as follows.
\begin{itemize}
\item The set ${\overline{D_{\{S_j\}_{j=1}^l}}}^{{\mathbb{R}}^2}-D_{\{S_j\}_{j=1}^l}$ is the disjoint union ${\sqcup}_{j=1}^l S_j$ of $l \geq 0$ curves. 
\item $S_j=\{x \in {\mathbb{R}}^2 \mid f_j(x)=0\}$, where $f_j:{\mathbb{R}}^2 \rightarrow \mathbb{R}$ is a smooth function such that $f_j {\mid}_{S_j}$ contains no critical point.
\item $X_{D,{\{S_j\}_{j=1}^l},m}$ is the zero set of some smooth map $f_{X_{D,{\{S_j\}_{j=1}^l},m}}:{\mathbb{R}}^{m_0} \rightarrow {\mathbb{R}}^{m_0-m}$ such that $S(f_{X_{D,{\{S_j\}_{j=1}^l},m}} {\mid}_{X_{D,{\{S_j\}_{j=1}^l},m}})$ is empty.
\item The image of the map ${\pi}_{m_0,m_0-m} {\mid}_{X_{D,{\{S_j\}_{j=1}^l},m}}$ is ${\overline{D_{\{S_j\}_{j=1}^l}}}^{{\mathbb{R}}^2}$. 
\item The relation ${\pi}_{m_0,m_0-m}(S({\pi}_{m_0,m_0-m} {\mid}_{X_{D,{\{S_j\}_{j=1}^l},m}}))={\sqcup}_{j=1}^l S_j$ also holds.
\end{itemize}
The case $(D_{\{S_j\}_{j=1}^1},S_1)=(D^2-S^1,S^1)$ with $X_{D,\{S_1\},m}:=S^m$ and $m_0=m+1$ is of simplest and certain generalized cases are first presented in \cite{kitazawa2}, followed by the author himself in the preprint \cite{kitazawa3}.
As presented above, the case $l=2$ with $S_i=\{(c_i(x),x) \mid x \in \mathbb{R}\}$ for some smooth functions $c_i:\mathbb{R} \rightarrow \mathbb{R}$ satisfying $c_2(x)-c_1(x)>0$ ($x \in \mathbb{R}$) is considered and of our central objects. For this, see the preprints \cite{kitazawa4, kitazawa5, kitazawa6, kitazawa7} again and see also \cite{kitazawa8} for example. Remember that the author has been interested in the (non-proper) functions ${\pi}_{m_0,1} {\mid}_{X_{D,{\{S_j\}_{j=1}^l},m}}$.
\subsection{Our main result.}
\label{subsec:1.4}
We present some of our main result.
\begin{Thm}
\label{thm:2}
Let $l=2$, $m \geq 2$ be an integer, $m_0=m+2$, $t_c<1$, and $c_{+}:\mathbb{R} \rightarrow \mathbb{R}$ be a smooth positive-valued function. 
\begin{enumerate}
\item \label{thm:2.1} We have $X_{D,{\{S_j\}_{j=1}^l},m}:=X_{c_{+},m}:=\{(x_1,x_2,t,(y_j)_{j=1}^{m-1}) \mid (1-t)(t-t_c)-{\Sigma}_{j=1}^{m-1} {y_j }^2=0, tc_{+}(x_1)-x_2=0\}$. ${\pi}_{m_0,1} {\mid}_{X_{c_{+},m}}$ is a proper submersion. %This is already presented in \cite{kitazawa5} as its main result.
\item \label{thm:2.2} For $c_{{+},m}:={\pi}_{m_0,1} {\mid}_{X_{c_{+},m}}$, ${\mathcal{R}}_{{c_{{+},m}}^{\prime,{\pi}_{m_0,1}}}$ is a $1$-dimensional CW complex. If $S(c_{+})$ is discrete and unbounded above and below in $\mathbb{R}$, then its vertex set is discrete if and only if for each connected component $I_{c_{+}}$ of $\mathbb{R}-S(c_{+})$, ${c_{{+},m}}^{\prime,{\pi}_{m_0,1}}(\{x \in I_{c_{+}}\mid {c_{+}}^{\prime \prime}(x)=0\})$ is a finite set, where ${c_{+}}^{\prime \prime}$  is the 2nd derivative of $c_{+}$.
\end{enumerate}
\end{Thm}
Theorem \ref{thm:2} (\ref{thm:2.2}) is satisfied for positive real analytic function $c_{+}$ such that $S(c_{+})$ is unbounded. In the next section, we discuss Theorem \ref{thm:2} and our related additional result (Theorems \ref{thm:6}, \ref{thm:7} and \ref{thm:8}).
\section{Our main result.}
We present the converse of Theorem \ref{thm:1}. Note that Theorems \ref{thm:1} and \ref{thm:3} are first presented as \cite[Theorem 3.1]{saeki1}.
\begin{Thm}[\cite{saeki1}]

\label{thm:3}
For a smooth function $c:X \rightarrow \mathbb{R}$ on a closed manifold such that the ${\mathcal{R}}_c$ is a graph, then $c(S(c))$ must be finite.
\end{Thm}
By arguments in the original article, we have Theorem \ref{thm:4}.
\begin{Thm}
\label{thm:4}
\begin{enumerate}
\item Theorems \ref{thm:1} and \ref{thm:3} hold for a proper smooth function $c:X \rightarrow \mathbb{R}$ on a manifold with no boundary, with $c(s(c))$ being replaced by closed in $\mathbb{R}$ and discrete and graphs being replaced by E-graphs.
\item For a proper smooth function $c:X \rightarrow \mathbb{R}$ on a manifold with no boundary such that $c(S(c))$ is not closed in $\mathbb{R}$ or that $c(S(c))$ is not discrete, there exists a point $s \in \mathbb{R}$ whose arbitrary neighborhood in $\mathbb{R}$ contains some points of $c(S(c))$ and for such a point $s$, there exists a point $s_0 \in {\bar{c}}^{-1}(s)$ whose arbitrary neighborhood contains some vertices of ${\mathcal{R}}_c$. 
\end{enumerate}
\end{Thm}
For Theorems \ref{thm:1}, \ref{thm:3} and \ref{thm:4}, see also \cite{saeki2}.

We prove Theorem \ref{thm:2}.
\begin{proof}[A proof of Theorem \ref{thm:2}]
We first prove the statement (\ref{thm:2.1}) in STEP 2-1. In STEP 2-2, we investigate ${c_{{+},m}}^{\prime,{\pi}_{m_0,1}}$ and $S({c_{{+},m}}^{\prime,{\pi}_{m_0,1}})$. These are review of some previous result of the author \cite{kitazawa7}. Last, in STEP 2-3, we prove the statement (\ref{thm:2.2}), as one of our essentially new result. \\
\ \\
STEP 2-1 The statement (\ref{thm:2.1}). \\

The statement (\ref{thm:2.1}) is already shown in \cite{kitazawa5} and we review this in a self-contained way.

First, the set $X_{c_{+},m}:=\{(x_1,x_2,t,(y_j)_{j=1}^{m-1}) \mid (1-t)(t-t_c)-{\Sigma}_{j=1}^{m-1} {y_j }^2=0, tc_{+}(x_1)-x_2=0\}$ is the zero set of the real polynomial map $F_{l,m,t_c,c_{+}}(x_{1},x_{2},t,{(y_{j})}_{j=1}^{m-1})=(F_{l,m,t_c,c_{+},1}(x_{1},x_{2},t,{(y_{j})}_{j=1}^{m-1}),F_{l,m,t_c,c_{+},2}(x_{1},x_{2},t,{(y_{j})}_{j=1}^{m-1})):=((1-t)(t-t_c)-{\Sigma}_{j=1}^{m-1} {y_j }^2, tc_{+}(x_1)-x_2)$ from ${\mathbb{R}}^{m_0}$ to ${\mathbb{R}}^2$. This is mapped onto a non-empty, open and connected set $D_{l,m,t_c,c_{+}}:=\{(x_1,x_2)\mid tc_{+}(x_1)-x_2=0, t_c<t<1\}$ of ${\mathbb{R}}^2$ by ${\pi}_{m_0,2}$.

We apply implicit function theorem to see that $X_{D,{\{S_j\}_{j=1}^l},m}:=X_{c_{+},m}$ is a desired $m$-dimensional smooth submanifold of ${\mathbb{R}}^{m_0}$ with no boundary.

At $(x_{1,0},x_{2,0},t_0,{(y_{j,0})}_{j=1}^{m-1}) \in X_{c_{+},m}$ with $(x_{1,0},x_{2,0}) \in D_{l,m,t_c,c_{+}}$, the value of the partial derivative $\frac{\partial F_{l,m,t_c,c_{+},1}}{\partial y_{j_0}}$ is non-zero for some $y_{j_0}$, and that of the partial derivative $\frac{\partial F_{l,m,t_c,c_{+},2}}{\partial y_{j}}$ is $0$. At the point the value of the partial derivative $\frac{\partial F_{l,m,t_c,c_{+},2}}{\partial t}$ is $c_{+}(x_1)>0$. At the point the value of the partial derivative $\frac{\partial F_{l,m,t_c,c_{+},1}}{\partial t}$ is non-zero if and only if $t_0 \neq 0$. We focus on some $y_{j_0}$ and $t$ for implicit function theorem.

At $(x_{1,0},x_{2,0},t_0,{(y_{j,0})}_{j=1}^{m-1}) \in X_{c_{+},m}$ with $(x_{1,0},x_{2,0}) \in {\overline{D_{l,m,t_c,c_{+}}}}^{{\mathbb{R}}^2}-D_{l,m,t_c,c_{+}}$, the value of the partial derivative $\frac{\partial F_{l,m,t_c,c_{+},1}}{\partial y_{j}}$ is $0$, the value of the partial derivative $\frac{\partial F_{l,m,t_c,c_{+},1}}{\partial t}$ is $0$, and that of the partial derivative $\frac{\partial F_{l,m,t_c,c_{+},1}}{\partial x_{j}}$ is $0$. At the point, the value of the partial derivative $\frac{F_{\partial l,m,t_c,c_{+},2}}{\partial t}$ is $c_{+}(x_1)>0$, that of the partial derivative $\frac{F_{l,m,t_c,c_{+},2}}{\partial x_2}$ is $-1$, and that of the partial derivative $\frac{F_{l,m,t_c,c_{+},2}}{\partial y_j}$ is $0$. We focus on some $x_2$ and $t$ for implicit function theorem.

We can apply implicit function theorem to see that $X_{D,{\{S_j\}_{j=1}^l},m}:=X_{c_{+},m} \subset {\mathbb{R}}^{m_0}$ is our desired submanifold.
By the construction, $c_{{+},m}:={\pi}_{m_0,1} {\mid}_{X_{c_{+},m}}$ is a proper submersion.
This completes the proof of the statement (\ref{thm:2.1}). \\
\ \\
STEP 2-2 ${c_{{+},m}}^{\prime,{\pi}_{m_0,1}}$ and $S({c_{{+},m}}^{\prime,{\pi}_{m_0,1}})$. \\

We investigate ${c_{{+},m}}^{\prime,{\pi}_{m_0,1}}$ and $S({c_{{+},m}}^{\prime,{\pi}_{m_0,1}})$. For this, it is sufficient to review arguments in \cite{kitazawa5} in a self-contained way.

We investigate the vector $v_{p,c_{{+},m},+}$, defined in the subsection \ref{subsec:1.2}, first. By our construction, this is presented as an $m_0$-dimensional vector of length $1$ of $T_p {\mathbb{R}}^{m_0}$ and canonically a real vector of $ {\mathbb{R}}^{m_0}$ the values of whose $i_p$-th components are $0$ for $i_p \geq 3$ and which is mapped to a vector tangent to the graph $\{(x_1,tc_{+}(x_1)) \mid x_1 \in \mathbb{R}\}$ by (the differential of) ${\pi}_{m_0,2}$.

From this, the value ${c_{{+},m}}^{\prime,{\pi}_{m_0,1}}(x_{1,0},x_{2,0},t_0,{(y_{j,0})}_{j=1}^{m-1})$ is shown to be $1$ if and only if at least one of the constraints holds.
\begin{itemize}
\item $t_0=0$.
\item $x_{1,0} \in S(c_{+})$.
\end{itemize}
We investigate the other points of $S({c_{{+},m}}^{\prime,{\pi}_{m_0,1}})$. For this, we investigate the case of $t \neq 0$ and $x_{1,0} \notin S(c_{+})$. \\

\noindent Case 2-2-1 The value ${c_{+}}^{\prime \prime}(x_{1,0})$ of the 2nd derivative ${c_{+}}^{\prime \prime}$ is non-zero. \\

In this case, we can immediately see $(x_{1,0},x_{2,0},t_0,{(y_{j,0})}_{j=1}^{m-1}) \notin S({c_{{+},m}}^{\prime,{\pi}_{m_0,1}})$. \\
\ \\
Case 2-2-2 The value ${c_{+}}^{\prime \prime}(x_{1,0})$ of the 2nd derivative ${c_{+}}^{\prime \prime}$ is $0$. \\

\noindent Case 2-2-2-1 $t_0 \neq 0,t_c,1$. \\

By our definition and the argument above, we have ${c_{{+},m}}^{\prime,{\pi}_{m_0,1}}(x_{1,0},x_{2,0},t_0,{(y_{j,0})}_{j=1}^{m-1})=\frac{1}{{t_0{c_{+}}^{\prime}(x_{1,0})}^2+1}$ and we have $(x_{1,0},x_{2,0},t_0,{(y_{j,0})}_{j=1}^{m-1}) \notin S({c_{{+},m}}^{\prime,{\pi}_{m_0,1}})$. \\ 
\ \\
Case 2-2-2-2 $t_0=t_c,1$ (with $t_c \neq 0$). \\

By our definition and the argument above, we can immediately see $(x_{1,0},x_{2,0},t_0,{(y_{j,0})}_{j=1}^{m-1}) \in S({c_{{+},m}}^{\prime,{\pi}_{m_0,1}})$. \\
\ \\
This completes STEP 2-2. This is presented as Theorem \ref{thm:5} again, later. \\
\ \\
STEP 2-3 The statement (\ref{thm:2.2}), one of our main ingredients of the present paper. \\

First, $X_{c_{+},m}-{{c_{{+},m}}^{\prime,{\pi}_{m_0,1}}}^{-1}(1)$ is, by STEP 2-2 (Theorem \ref{thm:5}), represented as the disjoint union of infinitely many non-empty subsets of $X_{c_{+},m}$ which are $m$-dimensional smooth manifolds with no boundary and whose closures taken in $X_{c_{+},m}$ are compact and connected. \\
\ \\
Case 2-3-1 The case $t_c \geq 0$.

Each subset can be denoted by $X_{c_{+},m,a_{j},a_{j+1}}$ with $j$ being an integer and $a_j, a_{j+1} \in c_{+}(S(c_{+}))$ satisfying the following.
\begin{itemize}
\item $a_{j}<a_{j+1}$.
\item ${\pi}_{m_0,2}({\overline{X_{c_{+},m,a_{j},a_{j+1}}}}^{X_{c_{+},m}})=(\{a_{j} \leq x_1 \leq a_{j+1}\} \times \mathbb{R}) \bigcap  {\overline{D_{l,m,t_c,c_{+}}}}^{{\mathbb{R}}^2}$ and ${\pi}_{m_0,2}(X_{c_{+},m,a_{j},a_{j+1}})=(\{a_j<x_1<a_{j+1}\} \times \mathbb{R}) \bigcap  {\overline{D_{l,m,t_c,c_{+}}}}^{{\mathbb{R}}^2}$.
\end{itemize}
Case 2-3-2 The case $t_c<0$.

Each subset can be denoted by $X_{c_{+},m,a_{j},a_{j+1},+}$ and $X_{c_{+},m,a_{j},a_{j+1},-}$ with $j$ being an integer and $a_j, a_{j+1} \in c_{+}(S(c_{+}))$ satisfying the following.
\begin{itemize}
\item $a_{j}<a_{j+1}$.
\item ${\pi}_{m_0,2}({\overline{X_{c_{+},m,a_{j},a_{j+1},+}}}^{X_{c_{+},m}})=(\{a_{j} \leq x_1 \leq a_{j+1}\} \times \{x_2 \geq 0\}) \bigcap  {\overline{D_{l,m,t_c,c_{+}}}}^{{\mathbb{R}}^2}$ and ${\pi}_{m_0,2}(X_{c_{+},m,a_{j},a_{j+1},+})=(\{a_j<x_1<a_{j+1}\} \times \{x_2>0\}) \bigcap  {\overline{D_{l,m,t_c,c_{+}}}}^{{\mathbb{R}}^2}$.
\item ${\pi}_{m_0,2}({\overline{X_{c_{+},m,a_{j},a_{j+1},-}}}^{X_{c_{+},m}})=(\{a_{j} \leq x_1 \leq a_{j+1}\} \times \{x_2 \leq 0\}) \bigcap  {\overline{D_{l,m,t_c,c_{+}}}}^{{\mathbb{R}}^2}$ and ${\pi}_{m_0,2}(X_{c_{+},m,a_{j},a_{j+1},-})=(\{a_j<x_1<a_{j+1}\} \times \{x_2<0\}) \bigcap  {\overline{D_{l,m,t_c,c_{+}}}}^{{\mathbb{R}}^2}$.
\end{itemize}

Furthermore, for each $0<r_1<r_2<1$, ${{c_{{+},m}}^{\prime,{\pi}_{m_0,1}}}^{-1}(\{r\mid r_1 \leq r \leq r_2\}) \bigcap {\overline{X_{c_{+},m,a_{j},a_{j+1}}}}^{X_{c_{+},m}}={{c_{{+},m}}^{\prime,{\pi}_{m_0,1}}}^{-1}(\{r\mid  r_1 \leq r \leq r_2\}) \bigcap X_{c_{+},m,a_{j},a_{j+1}}$ and the subset is compact. From this, the restriction ${\pi}_{m_0,1} {\mid}_{X_{c_{+},m,a_{j},a_{j+1}}}$ is proper. \\
\ \\
STEP 2-3-1 Around $X_{c_{+},m,a_{j},a_{j+1}}$ and $X_{c_{+},m,a_{j},a_{j+1},\pm}$. \\

Let $I_{c_{+}}:=\{a_j<x<a_{j+1}\}$ here. In the case $t_c \geq 0$, if the image ${c_{+}}^{\prime}(\{x \in I_{c_{+}}\mid {c_{+}}^{\prime \prime}(x)=0\})$ is finite, then ${\mathcal{R}}_{{c_{{+},m}}^{\prime,{\pi}_{m_0,1}} {\mid} X_{c_{+},m,a_{j},a_{j+1}}}$ is an E-graph from STEP 2-2 (Theorem \ref{thm:5}) and Theorem \ref{thm:4}. 
In the case $t_c<0$, if the image ${c_{+}}^{\prime}(\{x \in I_{c_{+}}\mid {c_{+}}^{\prime \prime}(x)=0\})$ is finite, then ${\mathcal{R}}_{{c_{{+},m}}^{\prime,{\pi}_{m_0,1}} {\mid} X_{c_{+},m,a_{j},a_{j+1},\pm}}$ are E-graphs from STEP 2-2 (Theorem \ref{thm:5}) and Theorem \ref{thm:4}. Hereafter, we name this condition that the image 
${c_{+}}^{\prime}(\{x \in I_{c_{+}}\mid {c_{+}}^{\prime \prime}(x)=0\})$ is finite a {\it finite-$I_{c_{+}}$} ({\it finite-$I_{c_{+},\pm}$}) condition in the case $t_c \geq 0$ (resp. $t_c<0$).

If the image ${c_{+}}^{\prime}(\{x \in I_{c_{+}}\mid {c_{+}}^{\prime \prime}(x)=0\})$ is not finite, then the following hold.
\begin{itemize}
\item There exists at least one number $r_{j,j+1}$ whose arbitrary neighborhood in $\mathbb{R}$ must contain some points of ${c_{+}}^{\prime}(\{x \in I_{c_{+}}\mid {c_{+}}^{\prime \prime}(x)=0\})$. We name the condition that we can choose $r_{j,j+1}=0$ an {\it $I_{c_{+}}$-to-$0$} ({\it $I_{c_{+},\pm}$-to-$0$}) condition in the case $t_c \geq 0$ (resp. $t_c<0$). %Note also that in the case $t_c<0$, if there exists $0<r_{j,j+1} \leq 1$ whose arbitrary neighborhood in $\mathbb{R}$ always contain some points of ${c_{{+},m}}^{\prime,{\pi}_{m_0,1}}(\{x \in I_{c_{+}}\mid {c_{+}}^{\prime \prime}(x)=0\})$.
\item In the case we have $r_{j,j+1}\neq 0$, the spaces ${\mathcal{R}}_{{c_{{+},m}}^{\prime,{\pi}_{m_0,1}} {\mid} X_{c_{+},m,a_{j},a_{j+1}}}$ and ${\mathcal{R}}_{{c_{{+},m}}^{\prime,{\pi}_{m_0,1}} {\mid} X_{c_{+},m,a_{j},a_{j+1},\pm}}$ are not E-graphs and their vertex sets are not discrete, due to Theorem \ref{thm:4} with STEP 2-2 (Theorem \ref{thm:5}). We need to discuss the conditions for $r_{j,j+1}=0$ above, in STEP 2-3-2, again.
\end{itemize}

Distinct $X_{c_{+},m,a_{j},a_{j+1}}$ and distinct $X_{c_{+},m,a_{j},a_{j+1},\pm}$ are mutually disjoint. In the case $t_c \geq 0$, distinct  ${\mathcal{R}}_{{c_{{+},m}}^{\prime,{\pi}_{m_0,1}} {\mid} X_{c_{+},m,a_{j},a_{j+1}}}$ are mutually disjoint. In the case $t_c<0$, distinct ${\mathcal{R}}_{{c_{{+},m}}^{\prime,{\pi}_{m_0,1}} {\mid} X_{c_{+},m,a_{j},a_{j+1},\pm}}$ are mutually disjoint. \\
\ \\
STEP 2-3-2 Around ${{c_{{+},m}}^{\prime,{\pi}_{m_0,1}}}^{-1}(1)$. \\

%Each point $p_{c_{+},1}$ of ${\bar{{{c_{{+},m}}^{\prime,{\pi}_{m_0,1}}}}}^{-1}(1)$ is a vertex of ${\mathcal{R}}_{{c_{{+},m}}^{\prime,{\pi}_{m_0,1}}}$. 
The set ${\bar{{{c_{{+},m}}^{\prime,{\pi}_{m_0,1}}}}}^{-1}(1)$ is a discrete set of countably many points in the case $t_c>0$ and a one-point set in the case $t_c \leq 0$.
In the case $t_c>0$, if a finite-$I_{c_{+}}$ condition is satisfied for all $I_{c_{+}}:=\{a_j<x<a_{j+1}\}$, then for each point $p_{c_{+},1}$ of ${\bar{{{c_{{+},m}}^{\prime,{\pi}_{m_0,1}}}}}^{-1}(1)$, we can have its small open neighborhood $U_{p_{c_{+},1}}$ in ${\mathcal{R}}_{{c_{{+},m}}^{\prime,{\pi}_{m_0,1}}}$ which is the image of the restriction of $q_{{c_{{+},m}}^{\prime,{\pi}_{m_0,1}}}$ to ${q_{{c_{{+},m}}^{\prime,{\pi}_{m_0,1}}}}^{-1}(U_{p_{c_{+},1}})$, and represented as a proper submersion with level sets consisting of two connected components, on ${q_{{c_{{+},m}}^{\prime,{\pi}_{m_0,1}}}}^{-1}(U_{p_{c_{+},1}})-{{c_{{+},m}}^{\prime,{\pi}_{m_0,1}}}^{-1}(1)$.

In the case $t_c=0$, if a finite-$I_{c_{+}}$ condition is satisfied for all $I_{c_{+}}:=\{a_j<x<a_{j+1}\}$, then for each point $p_{c_{+},1}$ of ${\bar{{{c_{{+},m}}^{\prime,{\pi}_{m_0,1}}}}}^{-1}(1)$, we can have its small open neighborhood $U_{p_{c_{+},1}}$ in ${\mathcal{R}}_{{c_{{+},m}}^{\prime,{\pi}_{m_0,1}}}$ which is the image of the restriction of $q_{{c_{{+},m}}^{\prime,{\pi}_{m_0,1}}}$ to ${q_{{c_{{+},m}}^{\prime,{\pi}_{m_0,1}}}}^{-1}(U_{p_{c_{+},1}})$, and represented as the disjoint union of (countably many) proper submersions whose level sets are connected, on ${q_{{c_{{+},m}}^{\prime,{\pi}_{m_0,1}}}}^{-1}(U_{p_{c_{+},1}})-{{c_{{+},m}}^{\prime,{\pi}_{m_0,1}}}^{-1}(1)$.

In the case $t_c<0$, if a finite-$I_{c_{+},\pm}$ condition is satisfied for all $I_{c_{+}}:=\{a_j<x<a_{j+1}\}$, then for each point $p_{c_{+},1}$ of ${\bar{{{c_{{+},m}}^{\prime,{\pi}_{m_0,1}}}}}^{-1}(1)$, we have its small open neighborhood $U_{p_{c_{+},1}}$ in ${\mathcal{R}}_{{c_{{+},m}}^{\prime,{\pi}_{m_0,1}}}$ which is the image of the restriction of ${c_{{+},m}}^{\prime,{\pi}_{m_0,1}}$ to ${q_{{c_{{+},m}}^{\prime,{\pi}_{m_0,1}}}}^{-1}(U_{p_{c_{+},1}})$, and represented as the disjoint union of (countably many) proper submersions whose level sets are connected, on ${q_{{c_{{+},m}}^{\prime,{\pi}_{m_0,1}}}}^{-1}(U_{p_{c_{+},1}})-{{c_{{+},m}}^{\prime,{\pi}_{m_0,1}}}^{-1}(1)$.

In the case $t_c \geq 0$, if an $I_{c_{+}}$-to-$0$ condition is satisfied for some $I_{c_{+}}:=\{a_j<x<a_{j+1}\}$, then for some point $p_{c_{+},1}$ of ${\bar{{{c_{{+},m}}^{\prime,{\pi}_{m_0,1}}}}}^{-1}(1)$, we cannot have any small open neighborhood $U_{p_{c_{+},1}}$ in ${\mathcal{R}}_{{c_{{+},m}}^{\prime,{\pi}_{m_0,1}}}$ which is the image of the restriction of $q_{{c_{{+},m}}^{\prime,{\pi}_{m_0,1}}}$ to ${q_{{c_{{+},m}}^{\prime,{\pi}_{m_0,1}}}}^{-1}(U_{p_{c_{+},1}})$, and represented as the disjoint union of proper submersions whose level sets are connected on ${q_{{c_{{+},m}}^{\prime,{\pi}_{m_0,1}}}}^{-1}(U_{p_{c_{+},1}})-{{c_{{+},m}}^{\prime,{\pi}_{m_0,1}}}^{-1}(1)$. In such cases, the vertex sets of the spaces are not discrete.

In the case $t_c< 0$, if an $I_{c_{+},\pm}$-to-$0$ condition is satisfied for some $I_{c_{+}}:=\{a_j<x<a_{j+1}\}$, then for some point $p_{c_{+},1}$ of ${\bar{{{c_{{+},m}}^{\prime,{\pi}_{m_0,1}}}}}^{-1}(1)$, we cannot have any small open neighborhood $U_{p_{c_{+},1}}$ in ${\mathcal{R}}_{{c_{{+},m}}^{\prime,{\pi}_{m_0,1}}}$ which is the image of the restriction of $q_{{c_{{+},m}}^{\prime,{\pi}_{m_0,1}}}$ to ${q_{{c_{{+},m}}^{\prime,{\pi}_{m_0,1}}}}^{-1}(U_{p_{c_{+},1}})$ and represented as the disjoint union of proper submersions whose level sets are connected on ${q_{{c_{{+},m}}^{\prime,{\pi}_{m_0,1}}}}^{-1}(U_{p_{c_{+},1}})-{{c_{{+},m}}^{\prime,{\pi}_{m_0,1}}}^{-1}(1)$. In such cases, the vertex sets of the spaces are not discrete.\\
\ \\
STEP 2-3-3 ${\mathcal{R}}_{{c_{{+},m}}^{\prime,{\pi}_{m_0,1}}}$ is Hausdorff and a $1$-dimensional CW complex, even if the additional assumption in Theorem \ref{thm:2} (\ref{thm:2.2}) is dropped. \\

${\mathcal{R}}_{{c_{{+},m}}^{\prime,{\pi}_{m_0,1}}}-{\bar{{{c_{{+},m}}^{\prime,{\pi}_{m_0,1}}}}}^{-1}(1)$ are Reeb spaces of proper submersions on manifolds with no boundary. 
Two distinct points in ${\mathcal{R}}_{{c_{{+},m}}^{\prime,{\pi}_{m_0,1}}}-{\bar{{{c_{{+},m}}^{\prime,{\pi}_{m_0,1}}}}}^{-1}(1)$ are separated by open subsets. 
This is due to several theory of Gelbukh such as \cite{gelbukh1, gelbukh3, gelbukh4}. 

For pairs of distinct points in ${\mathcal{R}}_{{c_{{+},m}}^{\prime,{\pi}_{m_0,1}}}$ of the other types, we have same observations by the arguments above. 

${\mathcal{R}}_{{c_{{+},m}}^{\prime,{\pi}_{m_0,1}}}$ is seen to be Hausdorff. By our construction, ${\mathcal{R}}_{{c_{{+},m}}^{\prime,{\pi}_{m_0,1}}}$ is always a $1$-dimensional cell complex. 
There, the closure of each of its vertices and edges always is shown to contain at most finitely many vertices and edges of it, by our arguments. This means that 
${\mathcal{R}}_{{c_{{+},m}}^{\prime,{\pi}_{m_0,1}}}$ is always a $1$-dimensional CW complex, even if the additional assumption in Theorem \ref{thm:2} (\ref{thm:2.2}) is dropped. \\
\ \\
From these arguments, this completes the proof of the statement (\ref{thm:2.2}) and STEP 2-3. \\
\ \\
This completes the proof.
\end{proof}

We summarize STEP 2-2 as Theorem \ref{thm:5}. This is also first presented as \cite[Theorem 5]{kitazawa5} and by the author.
\begin{Thm}
\label{thm:5}
In Theorem \ref{thm:2}, $S({c_{{+},m}}^{\prime,{\pi}_{m_0,1}})$ is the union of the following sets.
\begin{itemize}
\item The set of all points $(x_{1,0},x_{2,0},t_0,{(y_{j,0})}_{j=1}^{m-1}) \in X_{D,{\{S_j\}_{j=1}^l},m}$ satisfying at least one of the following. It is also the level set ${{c_{{+},m}}^{\prime,{\pi}_{m_0,1}}}^{-1}(1)$ of ${c_{{+},m}}^{\prime,{\pi}_{m_0,1}}$.
\begin{itemize}
\item $t_0=0$.
\item $x_{1,0} \in S(c_{+})$.
\end{itemize}
\item The value ${c_{+}}^{\prime \prime}(x_{1,0})$ of the 2nd derivative ${c_{+}}^{\prime \prime}$ is $0$ and $t_0 \in \{t_c,1\}$.
\end{itemize}
Furthermore, we have ${c_{{+},m}}^{\prime,{\pi}_{m_0,1}}(x_{1,0},x_{2,0},t_0,{(y_{j,0})}_{j=1}^{m-1})=\frac{1}{{t_0{c_{+}}^{\prime}(x_{1,0})}^2+1}$.
\end{Thm}
\begin{Rem}
\label{rem:1}
In Theorem \ref{thm:2}, the assumption that the set $S(c_{+})$ is discrete and unbounded above and below in $\mathbb{R}$ can be weaken in the following way.
\begin{itemize}
\item $S(c_{+})$ is unbounded.
\item At each point $p_{c_{+}} \in \mathbb{R}$, we can have its suitable open neighborhood $U_{p_{c_{+}}}$ in $\mathbb{R}$ which intersects at most finitely many connected components of $\mathbb{R}-S(c_{+})$ (and, with the number of connected components being $1$ in the case $p_{c_{+}} \in \mathbb{R}-S(c_{+})$ and $0$, $1$, or $2$, in the case $p_{c_{+}} \in S(c_{+})$, under this condition).
\end{itemize} 
\end{Rem}
Example \ref{ex:1} is an example (, which is a counterexample to a certain generalized variant of Theorem \ref{thm:2} and) related to Remark \ref{rem:1}.
\begin{Ex}
\label{ex:1}
In Theorem \ref{thm:2}, consider a positive number $r>0$ and the function $c_{+}$ defined by $c_{+}(x):=r+e^{-\frac{1}{x^2}}{\sin}^2 (\frac{1}{x})$ ($x \neq 0$) and $c_{+}(0)=r$. 
This is well-known to be a smooth function and real analytic on $\mathbb{R}-\{0\}$. We can check this as a kind of exercises. We may refer to the preprint \cite{kitazawa4} (\cite[A proof of Theorem 2]{kitazawa4}) and of course we do not assume related arguments of the preprint.
This does not satisfy the 2nd condition of Remark \ref{rem:1}. At $0 \in \mathbb{R}$, we cannot have its  open neighborhood $U_{p_{c_{+}}}$ in $\mathbb{R}$ which intersects at most finitely many connected components of $\mathbb{R}-S(c_{+})$.
\end{Ex}
Theorem \ref{thm:2} is improved as Theorem \ref{thm:6}.
\begin{Thm}
\label{thm:6}
We consider the situation of Theorem \ref{thm:2} {\rm (}Theorem \ref{thm:2} {\rm (}\ref{thm:2.1}{\rm )}{\rm )} and we also abuse the notation. Suppose that $S(c_{+})$ is unbounded above and below in $\mathbb{R}$. For $c_{{+},m}:={\pi}_{m_0,1} {\mid}_{X_{c_{+},m}}$, ${\mathcal{R}}_{{c_{{+},m}}^{\prime,{\pi}_{m_0,1}}}$ is a $1$-dimensional CW complex whose vertex set is discrete if and only if the following hold. 
 \begin{itemize}
\item {\rm (}Tameness for $S(c_{+})${\rm :} a condition respecting Remark \ref{rem:1} and weaker than the condition that $S(c_{+})$ is discrete.{\rm )} At each point $p_{c_{+}} \in \mathbb{R}$, we can have its suitable open neighborhood $U_{p_{c_{+}}}$ in $\mathbb{R}$ which intersects at most finitely many connected components of $\mathbb{R}-S(c_{+})$ is discrete.
\item {\rm(}Finite-$I_{c_{+}}$ conditions and finite-$I_{c_{+},\pm}$ conditions.{\rm )} For each connected component $I_{c_{+}}$ of $\mathbb{R}-S(c_{+})$, ${c_{+}}^{\prime}(\{x \in I_{c_{+}}\mid {c_{+}}^{\prime \prime}(x)=0\})$ is a finite set.

\end{itemize}
Furthermore, the following hold in the case the conditions above hold.
\begin{enumerate}
\item In the case $t_c>0$, ${\mathcal{R}}_{{c_{{+},m}}^{\prime,{\pi}_{m_0,1}}}$ is a graph which is also infinite.
\item In the case $t_c \leq 0$, ${\mathcal{R}}_{{c_{{+},m}}^{\prime,{\pi}_{m_0,1}}}$ is a $1$-dimensional CW complex whose vertex set is discrete and which has exactly one vertex whose neighborhood in ${\mathcal{R}}_{{c_{{+},m}}^{\prime,{\pi}_{m_0,1}}}$  is always non-compact. 
\end{enumerate}
\end{Thm}
\begin{proof}

We prove the former statement and after that we prove the latter one. We apply arguments in "A proof of Theorem \ref{thm:2}" with Remark \ref{rem:1}, implicitly. \\
\ \\
STEP 6-1 The former statement. \\

The former statement is explicitly discussed by the observation as in Remark \ref{rem:1} (Example \ref{ex:1}). 

Suppose that the condition "Tameness for $S(c_{+})$" is dropped and we discuss a contradiction.

Consider a point as $p_0 \in \mathbb{R}$ discussed in Example \ref{ex:1}. We have $p_0 \in S(c_{+})$. \\
\ \\
Case 6-1-1 The case $t_c>0$. \\

Each open set of $X_{c_{+},m}$ which is, by ${\pi}_{m_0,1}$, mapped onto a subset in $\mathbb{R}$ regarded to be a neighborhood of $p_0$ in $\mathbb{R}$ must intersect infinitely many connected component of ${{c_{{+},m}}^{\prime,{\pi}_{m_0,1}}}^{-1}(1)$, due to Theorem \ref{thm:5}. \\
\ \\
Case 6-1-2 The case $t_c \leq 0$. \\

By our definitions, we have a sequence $\{p_{1,2,j}\}_{j=1}^{\infty} \subset \mathbb{R}$ as follows.
\begin{itemize}
\item $p_{1,2,j} \in \mathbb{R}-S(c_{+})$.
\item ${c_{+}}^{\prime \prime}(p_{1,2,j})=0$.
\item The sequence converges to $p_0$.
\end{itemize}

In this case, the level set ${{c_{{+},m}}^{\prime,{\pi}_{m_0,1}}}^{-1}(1)$ of ${c_{{+},m}}^{\prime,{\pi}_{m_0,1}}$ is a non-empty connected set and ${\bar{{c_{{+},m}}^{\prime,{\pi}_{m_0,1}}}}^{-1}(1)$ is a one-point set.
From Theorem \ref{thm:5}, each open neighborhood of ${\bar{{c_{{+},m}}^{\prime,{\pi}_{m_0,1}}}}^{-1}(1)$ in ${\mathcal{R}}_{{c_{{+},m}}^{\prime,{\pi}_{m_0,1}}}$ must contain infinitely many vertices in ${\mathcal{R}}_{{c_{{+},m}}^{\prime,{\pi}_{m_0,1}}}$ which are also realized as the values of the quotient map $q_{{c_{{+},m}}^{\prime,{\pi}_{m_0,1}}}$ at points such that in ($(x_{1,0},x_{2,0},t_0,{(y_{j,0})}_{j=1}^{m-1}) \in X_{D,{\{S_j\}_{j=1}^l},m}$ of) Theorem \ref{thm:5}, ${c_{+}}^{\prime \prime}(x_{1,0})=0$, that $x_{1,0} \in \mathbb{R}-S(c_{+})$ and that $t_0 \in \{t_c,1\}$.

${\mathcal{R}}_{{c_{{+},m}}^{\prime,{\pi}_{m_0,1}}}$ is shown to be a $1$-dimensional CW complex whose vertex set is not discrete. \\
\ \\
STEP 6-2 The latter statement. \\
\ \\
The latter statement is discussed by the story of "A proof of Theorem \ref{thm:2}". In the case $t_c>0$, the set ${\bar{{{c_{{+},m}}^{\prime,{\pi}_{m_0,1}}}}}^{-1}(1)$ is discrete in ${\mathcal{R}}_{{c_{{+},m}}^{\prime,{\pi}_{m_0,1}}}$ and each point there is a vertex of degree $2$ in ${\mathcal{R}}_{{c_{{+},m}}^{\prime,{\pi}_{m_0,1}}}$. In the case $t_c \leq 0$, the set ${\bar{{{c_{{+},m}}^{\prime,{\pi}_{m_0,1}}}}}^{-1}(1)$ is a one-point set in ${\mathcal{R}}_{{c_{{+},m}}^{\prime,{\pi}_{m_0,1}}}$ and this is the unique vertex in ${\mathcal{R}}_{{c_{{+},m}}^{\prime,{\pi}_{m_0,1}}}$ whose neighborhood there is always non-compact. \\
\ \\
This completes the proof.
\end{proof}

We consider a case where $S_{c_{+}}$ may not be unbounded.
\begin{Thm}
\label{thm:7}
We consider the situation of Theorem \ref{thm:2} {\rm (}Theorem \ref{thm:2} {\rm (}\ref{thm:2.1}{\rm )}{\rm )} and we also abuse the notation. Suppose that the 1st derivative ${c_{+}}^{\prime}(x)$ converges to $0$
as $x$ diverges to $\pm \infty$. For $c_{{+},m}:={\pi}_{m_0,1} {\mid}_{X_{c_{+},m}}$, ${\mathcal{R}}_{{c_{{+},m}}^{\prime,{\pi}_{m_0,1}}}$ is a $1$-dimensional CW complex whose vertex set is discrete if and only if the following hold. 
 \begin{itemize}
\item At each point $p_{c_{+}} \in \mathbb{R}$, we can have its suitable open neighborhood $U_{p_{c_{+}}}$ in $\mathbb{R}$ which intersects at most finitely many connected components of $\mathbb{R}-S(c_{+})$.
\item {\rm(}Finite-$I_{c_{+}}$ conditions and finite $I_{c_{+},\pm}$ conditions.{\rm )} For each connected component $I_{c_{+}}$ of $\mathbb{R}-S(c_{+})$, ${c_{+}}^{\prime}(\{x \in I_{c_{+}}\mid {c_{+}}^{\prime \prime}(x)=0\})$ is a finite set.

\end{itemize}
\end{Thm}
\begin{proof}
It is sufficient to discuss the case $S_{c_{+}}$ is not unbounded above or is not unbounded below.
In this case, we have one or two connected components of $\mathbb{R}-S(c_{+})$. We review "A proof of Theorem \ref{thm:2}" and here $X_{c_{+},m}-{{c_{{+},m}}^{\prime,{\pi}_{m_0,1}}}^{-1}(1)$ contains one, two or four $m$-dimensional connected manifolds with no boundary whose closures in $X_{c_{+},m}$ are not bounded, and are non-compact, in ${\mathbb{R}}^{m_0}$. 
Let such a connected manifold in $X_{c_{+},m}-{{c_{{+},m}}^{\prime,{\pi}_{m_0,1}}}^{-1}(1)$ be denoted by $X_{c_{+},m,\infty}$. In the case $t_c \geq 0$ ($t_c<0$), we can also consider {\it finite-$I_{c_{+}}$} (resp. {\it finite-$I_{c_{+},\pm}$}) conditions and {\it $I_{c_{+}}$-to-$0$} (resp. {\it $I_{c_{+},\pm}$-to-$0$}) conditions as in "A proof of Theorem \ref{thm:2}".

The 1st derivative ${c_{+}}^{\prime}(x)$ converges to $0$ as $x$ diverges to $\pm \infty$.

By Theorem \ref{thm:5} and the calculation ${c_{{+},m}}^{\prime,{\pi}_{m_0,1}}(x_{1,0},x_{2,0},t_0,{(y_{j,0})}_{j=1}^{m-1})=\frac{1}{{t_0{c_{+}}^{\prime}(x_{1,0})}^2+1}$ there, the intersection ${{c_{{+},m}}^{\prime,{\pi}_{m_0,1}}}^{-1}(\{r \mid 0<r_1 \leq r \leq r_2<1\}) \bigcap X_{c_{+},m,\infty}$ of the following two is bounded in ${\mathbb{R}}^{m_0}$ and compact.
\begin{itemize}
\item The preimage ${{c_{{+},m}}^{\prime,{\pi}_{m_0,1}}}^{-1}(\{r \mid 0<r_1 \leq r \leq r_2<1\})$ of the set $\{r \mid 0<r_1 \leq r \leq r_2 <1\} \subset \mathbb{R}$ with $0<r_1<r_2<1$.
\item A connected manifold $X_{c_{+},m,\infty}$ in $X_{c_{+},m}-{{c_{{+},m}}^{\prime,{\pi}_{m_0,1}}}^{-1}(1)$, whose closure in $X_{c_{+},m}$ is not bounded, and is non-compact in ${\mathbb{R}}^{m_0}$.
\end{itemize}
In addition, the function ${c_{{+},m}}^{\prime,{\pi}_{m_0,1}} {\mid}_{{{c_{{+},m}}^{\prime,{\pi}_{m_0,1}}}^{-1}(\{r \mid 0<r_1 \leq r \leq r_2<1\}) \bigcap X_{c_{+},m,\infty}}$ is proper.

Except the presented arguments, we can argue as in "A proof of Theorem \ref{thm:2}" to complete the proof. \end{proof}
For example, positive-valued real analytic functions with $S(c_{+})$ and $S({c_{+}}^{\prime})$ being bounded are examples for Theorem \ref{thm:7}. We give a counterexample to Theorem \ref{thm:7}, with an $I_{c_{+}}$-to-$0$ ($I_{c_{+},\pm}$-to-$0$) condition for an open connected set $I_{c_{+}}$ which is not bounded.
\begin{Ex}
\label{ex:2}
For a positive number $r>0$,
let $c_{+}(x):={\int}_{0}^{x} \frac{1}{2}e^{-t^2}(1+{\sin}^2 (e^{t^4})) dt$. The 1st derivative ${c_{+}}^{\prime}(x)$ converges to $0$ as $x$ diverges to $\pm \infty$.
The 2nd derivative is calculated to be ${c_{+}}^{\prime \prime}(x)=-xe^{-x^2}(1+{\sin}^2 (e^{x^4}))+2x^3e^{x^4-x^2}\sin (2e^{x^4})$. By the form, 
we have ${\limsup}_{x \to \pm \infty} c_{+}(x)=+\infty $ and ${\liminf}_{x \to \pm \infty} c_{+}(x)=-\infty$. 
$S({c_{+}}^{\prime})$ is unbounded above and below. For this kind of functions, see \cite{kitazawa6, kitazawa7}. This has been used for cases of the subsection \ref{subsec:1.3}, originally, by the author.

This case is a counterexample to Theorem \ref{thm:7}. 
\end{Ex} 

\begin{Thm}
\label{thm:8}
We consider the situation of Theorem \ref{thm:2} {\rm (}Theorem \ref{thm:2} {\rm (}\ref{thm:2.1}{\rm )}{\rm )} and we also abuse the notation. Suppose the following.

\begin{itemize}
\item $S(c_{+})$ is not empty. It is not unbounded above or is not unbounded below {\rm :} we have at least one connected component $I_{c_{+}}$ of $\mathbb{R}-S(c_{+})$ which is not bounded.
\item For each connected component $I_{c_{+}}$ of $\mathbb{R}-S(c_{+})$ which is not bounded, for the absolute values $|{c_{+}}^{\prime}(x)|$ of ${c_{+}}^{\prime}(x)$, it always holds that $|{c_{+}}^{\prime}(p_1)| \leq |{c_{+}}^{\prime}(p_2)|$ for arbitrary pair $(p_1,p_2)$ of numbers in $I_{c_{+}}$ satisfying the following{\rm :} the distance between $p_1$ and $S(c_{+})$ in $\mathbb{R}$ is smaller than that between $p_2$ and $S(c_{+})$. 
\end{itemize}
%In this situation, $X_{c_{+},m}-{{c_{{+},m}}^{\prime,{\pi}_{m_0,1}}}^{-1}(1)$ contains at least one connected manifold with no boundary whose closure in $X_{c_{+},m}$ is not bounded and is non-compact in ${\mathbb{R}}^{m_0}$. 

 In this situation, for $c_{{+},m}:={\pi}_{m_0,1} {\mid}_{X_{c_{+},m}}$, ${\mathcal{R}}_{{c_{{+},m}}^{\prime,{\pi}_{m_0,1}}}$ is a $1$-dimensional CW complex whose vertex set is discrete if and only if the following hold. 
\begin{itemize}
\item At each point $p_{c_{+}} \in \mathbb{R}$, we can have its suitable open neighborhood $U_{p_{c_{+}}}$ in $\mathbb{R}$ which intersects at most finitely many connected components of $\mathbb{R}-S(c_{+})$ is discrete.
\item {\rm(}Finite-$I_{c_{+}}$ conditions and finite $I_{c_{+},\pm}$ conditions.{\rm )} For each connected component $I_{c_{+}}$ of $\mathbb{R}-S(c_{+})$ which is bounded, ${c_{{+},m}}^{\prime,{\pi}_{m_0,1}}(\{x \in I_{c_{+}}\mid {c_{+}}^{\prime \prime}(x)=0\})$ is a finite set.
\item  {\rm (}Tameness of ${c_{+}}^{\prime}(\{x \in I_{c_{+}}\mid {c_{+}}^{\prime \prime}(x)=0\})$ in the set ${c_{+}}^{\prime}({\overline{I_{c_{+}}}}^{\mathbb{R}})$ for connected components $I_{c_{+}}$ of $\mathbb{R}-S(c_{+})$ which are not bounded.{\rm )} For each connected component $I_{c_{+}}$ of $\mathbb{R}-S(c_{+})$ which is not bounded, the image ${c_{+}}^{\prime}(\{x \in I_{c_{+}}\mid {c_{+}}^{\prime \prime}(x)=0\})$ is discrete and closed in the set ${c_{+}}^{\prime}({\overline{I_{c_{+}}}}^{\mathbb{R}})$.
\end{itemize}
\end{Thm}
\begin{proof}
It is sufficient to discuss 
"Tameness of ${c_{+}}^{\prime}(\{x \in I_{c_{+}}\mid {c_{+}}^{\prime \prime}(x)=0\})$ in the set ${c_{+}}^{\prime}({\overline{I_{c_{+}}}}^{\mathbb{R}})$ for connected components $I_{c_{+}}$ of $\mathbb{R}-S(c_{+})$ which are not bounded". Here $X_{c_{+},m}-{{c_{{+},m}}^{\prime,{\pi}_{m_0,1}}}^{-1}(1)$ contains one, two or four $m$-dimensional connected manifolds with no boundary whose closures in $X_{c_{+},m}$ are not bounded, and are non-compact, in ${\mathbb{R}}^{m_0}$. 
Let such a connected manifold in $X_{c_{+},m}-{{c_{{+},m}}^{\prime,{\pi}_{m_0,1}}}^{-1}(1)$ be denoted by $X_{c_{+},m,\infty}$. In the case $t_c \geq 0$ ($t_c<0$), we can also consider {\it finite-$I_{c_{+}}$} (resp. {\it finite-$I_{c_{+},\pm}$}) conditions and {\it $I_{c_{+}}$-to-$0$} (resp. {\it $I_{c_{+},\pm}$-to-$0$}) conditions.

By Theorem \ref{thm:5} and the calculation ${c_{{+},m}}^{\prime,{\pi}_{m_0,1}}(x_{1,0},x_{2,0},t_0,{(y_{j,0})}_{j=1}^{m-1})=\frac{1}{{t_0{c_{+}}^{\prime}(x_{1,0})}^2+1}$ there, the intersection ${{c_{{+},m}}^{\prime,{\pi}_{m_0,1}}}^{-1}(r_0) \bigcap X_{c_{+},m,\infty}$ of the following two is connected and is not bounded in ${\mathbb{R}}^{m_0}$.
\begin{itemize}
\item The level set ${{c_{{+},m}}^{\prime,{\pi}_{m_0,1}}}^{-1}(r_0)$ of the function ${c_{{+},m}}^{\prime,{\pi}_{m_0,1}}$ is connected and homeomorphic to $\{t \mid t \geq 0\}$ for each number $0<r_0<1$.
\item A connected manifold $X_{c_{+},m,\infty}$ in $X_{c_{+},m}-{{c_{{+},m}}^{\prime,{\pi}_{m_0,1}}}^{-1}(1)$, whose closure in $X_{c_{+},m}$ is not bounded, and is non-compact in ${\mathbb{R}}^{m_0}$.
\end{itemize}

Remember another assumption. It is assumed that for each connected component $I_{c_{+}}$ of $\mathbb{R}-S(c_{+})$ which is not bounded, for the absolute values $|{c_{+}}^{\prime}(x)|$ of ${c_{+}}^{\prime}(x)$, it always holds that $|{c_{+}}^{\prime}(p_1)| \leq |{c_{+}}^{\prime}(p_2)|$ for $p_1,p_2 \in I_{c_{+}}$ with the following: the distance between $p_1$ and $S(c_{+})$ is smaller than that between $p_2$ and $S(c_{+})$. Due to this, with our construction and arguments, the function $\bar{{c_{{+},m}}^{\prime,{\pi}_{m_0,1}} {\mid}_{{\mathcal{R}}_{{c_{{+},m}}^{\prime,{\pi}_{m_0,1}} {\mid}_{X_{c_{+},m,\infty}}}}}:{\mathcal{R}}_{{c_{{+},m}}^{\prime,{\pi}_{m_0,1}} {\mid}_{X_{c_{+},m,\infty}}} \rightarrow \bar{{c_{{+},m}}^{\prime,{\pi}_{m_0,1}} {\mid}_{X_{c_{+},m,\infty}}}({\mathcal{R}}_{{c_{{+},m}}^{\prime,{\pi}_{m_0,1}} {\mid}_{X_{c_{+},m,\infty}}}) \subset \mathbb{R}$ is shown to be a continuous bijection onto the space of the target of the form $\{r \mid r_0<r<1\}$, or $\{r \mid r_0 \leq r<1\}$ with some non-negative number $0 \leq r_0<1$, and an open map onto the subspace of $\mathbb{R}$. By \cite[Theorem 6.2]{gelbukh4}, this is a homeomorphism. 
For vertices of ${\mathcal{R}}_{{c_{{+},m}}^{\prime,{\pi}_{m_0,1}} {\mid}_{X_{c_{+},m,\infty}}}$, apply Theorem \ref{thm:5}, with Theorem \ref{thm:4}.

Last, additional arguments as in proofs or our previous theorems such as Theorems \ref{thm:2}, \ref{thm:6}, and \ref{thm:7}, complete the proof.
 \end{proof}
 We present an example to Theorem \ref{thm:8} as Example \ref{ex:3}.
 \begin{Ex}
 	\label{ex:3}
 	Let $r_{\rm P}>0$ be a positive number. Let $c_{+,0}:\mathbb{R} \rightarrow \mathbb{R}$ be a real analytic function such that $c_{+,0}(x) \geq 0$ for $x \in \mathbb{R}$, and that the zero set of $c_{+,0}$ is not bounded. For example, consider the case $c_{+,0}(x)={\sin}^2 x$. Let
 	$c_{+}(x):=r_{\rm P}+{\int}_{0}^x ({\int}_{0}^{x_a} tc_{+,0}(t) dt)d_{x_a}$. The critical set $S(c_{+})$ of $c_{+}$ is $\{0\}$. The critical set $S({c_+}^{\prime})$ of $c^{\prime}$ is not bounded. The function $\bar{{c_{{+},m}}^{\prime,{\pi}_{m_0,1}} {\mid}_{{\mathcal{R}}_{{c_{{+},m}}^{\prime,{\pi}_{m_0,1}} {\mid}_{X_{c_{+},m,\infty}}}}}:{\mathcal{R}}_{{c_{{+},m}}^{\prime,{\pi}_{m_0,1}} {\mid}_{X_{c_{+},m,\infty}}} \rightarrow \bar{{c_{{+},m}}^{\prime,{\pi}_{m_0,1}} {\mid}_{X_{c_{+},m,\infty}}}({\mathcal{R}}_{{c_{{+},m}}^{\prime,{\pi}_{m_0,1}} {\mid}_{X_{c_{+},m,\infty}}}) \subset \mathbb{R}$ is a continuous bijection onto the space of the target of the form $\{r \mid r_0<r<1\}$. This gives an example for Theorem \ref{thm:8}.
 	
 	\end{Ex}
 	\begin{Rem}
 		In the assumption of Theorem \ref{thm:8}, for the absolute values "$|{c_{+}}^{\prime}(p_1)|$" and  "$|{c_{+}}^{\prime}(p_2)|$", we have $|{c_{+}}^{\prime}(p_i)|={c_{+}}^{\prime}(p_i)>0$. This is due to the assumption that $c_{+}(x)$ is always positive.
 	\end{Rem}
\section{Conflict of interest and Data availability.}
The author is a researcher at Osaka Central Advanced Mathematical Institute (OCAMI researcher). The institution is funded by MEXT Promotion of Distinctive Joint Research Center Program JPMXP0723833165. He thanks all there for the hospitality, where he is not employed by the institute or projects there hosted by some members. 
  %Some of works by other researchers and this version may overlap in some of the contents due to the nature that our problems are natural in theory of Morse functions and applications to differential topology and that related mathematical studies are very fundamental and classical in some senses, for example. However the present version of our paper is presented independent of these work. \\
  %Saga Souhatsu Mathematical Seminar (http://inasa.ms.saga-u.ac.jp/Japanese/saga-souhatsu.html), inviting the author as a speaker, is funded and supported by JST Fusion Oriented REsearch for disruptive Science and Technology JPMJFR202U: the author was a speaker on 2024/7/12 supported by this project.\\
 
No data other than the present file is generated, related to the present paper. We do not assume non-trivial arguments in preprints which are still unpublished formally. We may refer to them to some extent.

\end{document}